\documentclass[11pt,reqno]{amsart}
\usepackage{amssymb,amsmath,amsthm,amsfonts}
\usepackage[english]{babel}
\usepackage{mathrsfs,a4wide,dsfont}
\usepackage[utf8]{inputenc}
\usepackage{color}
\usepackage{comment}
\usepackage[english]{babel}
\usepackage[normalem]{ulem} 
\usepackage{hyperref}

\theoremstyle{plain}
\newtheorem{theorem}{Theorem}[section]
\newtheorem{lemma}[theorem]{Lemma}
\newtheorem{proposition}[theorem]{Proposition}

\newtheorem{remark}[theorem]{Remark}

\theoremstyle{definition}
\theoremstyle{remark}
\numberwithin{equation}{section}
\newenvironment{oss}{\begin{remark} \begin{rm}}{\end{rm} \end{remark}}

\renewcommand{\O}{\Omega}
\renewcommand{\o}{\omega}

\newcommand{\R}{{\mathbb R}}
\newcommand{\N}{{\mathbb N}}

\newcommand{\weakst}{\stackrel{\ast}{\rightharpoonup}}
\newcommand{\weak}{\rightharpoonup}

\def\Div{\textup{div}\,}
\def\Curl{\textup{curl}\,}

\newcommand{\e}{\varepsilon}

\newcommand{\tu}{\widetilde{u}}
\newcommand{\tv}{\widetilde{v}}
\newcommand{\sol}{{\rm sol}}
\newcommand{\pot}{{\rm pot}}

\usepackage{xcolor}

    \let\TeXchi\chi
\newbox\chibox
\setbox0 \hbox{\mathsurround0pt $\TeXchi$}
\setbox\chibox \hbox{\raise\dp0 \box 0 }
\def\chi{\copy\chibox}

\title[Stochastic homogenization of maximal monotone relations]{Stochastic homogenization of maximal monotone relations and applications}
\author[L.\,Lussardi]
{Luca Lussardi}
\address[Luca Lussardi]{Dipartimento di Scienze Matematiche ``G.L.\,Lagrange'', Politecnico di Torino, c.so Duca degli Abruzzi 24, I-10129 Torino, Italy}
\email[]{luca.lussardi@polito.it}
\urladdr{http://www.dmf.unicatt.it/~lussardi/}
\author[S.\,Marini]
{Stefano Marini}
\address[Stefano Marini]{Dipartimento di Matematica e Fisica ``N.\,Tartaglia'', Universit\`a Cattolica del Sacro Cuore, via dei Musei 41, I-25121 Brescia, Italy}
\email[]{stefano.marini02@icatt.it}
\author[M.\,Veneroni]
{Marco Veneroni}
\address[Marco Veneroni]{Dipartimento di Matematica ``F.\,Casorati", Universit\`a degli Studi di Pavia, via Ferrata 5, I-27100 Pavia, Italy}
\email[]{marco.veneroni@unipv.it}
\urladdr{http://www-dimat.unipv.it/~veneroni}

\begin{document}

\baselineskip3.4ex

\vspace{0.5cm}
\begin{abstract}
We study the homogenization of a stationary random maximal monotone operator on a probability space equipped with an ergodic dynamical system. The proof relies on Fitzpatrick's variational formulation of monotone relations, on Visintin's scale integration/disintegration theory and on Tartar-Murat's compensated compactness. We provide applications to systems of PDEs with random coefficients arising in electromagnetism and in nonlinear elasticity.

\vskip .3truecm

\end{abstract}



%


\date{\today}

\maketitle


\section{Introduction}


Stochastic homogenization is a broadly studied subject, starting from the seminal papers by Kozlov \cite{K} and Papanicolaou-Varadhan \cite{PV}, who studied boundary value problems for second order linear PDEs.  We prove here an abstract homogenization result for the graph of a random maximal monotone operator 
\[ 
	v(x,\o) \in \alpha_\e(x,\o,u(x,\o)),
\]
where $x \in \R^n$ and $\omega$ is a parameter in a probability space $\O$. In the spirit of \cite[Chapter 7]{JikovKozlovOleinik94}, the random operator $\alpha_\e$ is obtained from a stationary operator $\alpha$ via an ergodic dynamical system $T_x:\Omega \to \Omega$
\begin{equation}
\label{eq:alphaintro}
		\alpha_\e(x,\o,\cdot):=  \alpha\left(T_{x/\e}\o,\cdot\right).
\end{equation}
The aim of this paper is to extend existing results where $\alpha$ is the subdifferential of a convex function \cite{Veneroni11, SchweizerHysPM} to general maximal monotone operators and to provide a simple proof based on Tartar's oscillating test function method. The crucial ingredient in our analysis is the scale integration/disintegration theory introduced by Visintin \cite{Visintin09}. Moreover, relying on Fitzpatrick's variational formulation of monotone graphs \cite{Fitzpatrick}, which perfectly suits the scale integration/disintegration setting \cite{VisintinAsAn2013}, in the proof we can directly exploit the maximal monotonicity, without turning to (stochastic) two-scale convergence \cite{BourgeatMikelicWright94, SangoWoukeng, HeidaNesenenko}, $\Gamma$-convergence \cite{DalmasoModica86b,DalmasoModica86a}, G-convergence \cite{Pankov1991,Pankov1997}, nor epigraph convergence \cite{MessaoudiMichaille91, MessaoudiMichaille94}. An advantage of our approach is that we don't need to assume an additional compact metric space structure on the probability space $\O$. Moreover, the effective relation is obtained directly, i.e.\,without an intermediate two-scale problem, which often needs to be studied separately. We also obtain the existence of the oscillating test functions as a byproduct of scale disintegration, without having to study the auxiliary problem (see, e.g., \cite[Section 3.2]{Pankov1997}. 

The outline of the proof is the following: Let  $X$ be a separable and reflexive Banach space, with dual $X'$, let $A_n:X \to X'$ be a sequence of monotone operators, and let $(x_n,y_n)\in X \times X'$ be a sequence of points on the graphs of $A_n$, i.e., such that $y_n=A_n x_n$ for all $n \in \N$. Assuming that $(x_n,y_n) \weak (x,y)$ in $X \times X'$ and that we already know the limit maximal monotone operator $A:X \to X'$, a classical question of functional analysis is: 

\centerline{\emph{Under which assumptions can we conclude that $y=Ax$?}} 
A classical answer (see, e.g., \cite{Brezis73}) is: If we can produce an auxiliary sequence of points on the graph of $A_n$, and we know that they converge to a point on the graph of $A$, that is, if there exist
\begin{equation}
\label{eq:schemeaux}
	(\xi_n,\eta_n)\in X \times X'\quad \text{such that}\quad \eta_n = A_n \xi_n,\quad (\xi_n,\eta_n) \weak (\xi,\eta), \quad \text{and}\quad \eta = A \xi,
\end{equation}
then, denoting by $\langle y,x \rangle$ the duality between $x\in X$ and $y\in X'$, by monotonicity of $A_n$
\[
	\langle y_n - \eta_n,x_n-\xi_n\rangle \geq 0.
\]
In order to pass to the limit as $n \to \infty$, since the duality of weak converging sequences in general does not converge to the duality of the limit, we need the additional hypothesis
\begin{equation}
\label{eq:schemecomp}
	\limsup_{n \to \infty}\, \langle g_n ,f_n\rangle \leq\langle g,f \rangle \qquad \forall\, (f_n,g_n)\weak (f,g)\ \text{in }X \times X',
\end{equation}
which, together with the weak convergence of $(x_n,y_n)$ and $(\xi_n,\eta_n)$, yields
\[
	\langle y - \eta,x-\xi\rangle \geq 0.
\]
By \emph{maximal monotonicity} of $A$, if the last inequality is satisfied for all $(\xi,\eta)$ such that $\eta=A\xi$, then we can conclude that  $(x,y)$ is a point of the graph of $A$, i.e., $y=Ax$. Summarizing, this procedure is based on
\begin{enumerate}
	\item Existence and weak compactness of solutions $(x_n,y_n)$ such that $y_n=A_n x_n$ and $(x_n,y_n) \weak (x,y)$;
	\item A condition for the convergence of the duality product \eqref{eq:schemecomp};
	\item Existence of a recovery sequence \eqref{eq:schemeaux} for all points in the limit graph.
\end{enumerate}
The first step depends on the well-posedness of the application; the second step is ensured, e.g., by compensated compactness (in the sense of Murat-Tartar \cite{Murat78,Tartar77}), and, like the first one, it depends on the character of the differential operators that appear in the application, rather than on the homogenization procedure. In the present paper we focus on the third step: in the context of stochastic homogenization, we prove that the scale integration/disintegration idea introduced by Visintin \cite{Visintin09}, combined with Birkhoff's ergodic theorem (Theorem \ref{B}) yields the desired recovery sequence. We obtain an explicit formula for the limit operator $A$ through the scale integration/disintegration procedure with Fitzpatrick's variational formulation. With the notation introduced in \eqref{eq:alphaintro}, the outline of this procedure is the following:
\[
	\alpha\ \stackrel{a)}{\longrightarrow}\ f\ \stackrel{b)}{\longrightarrow}\ f_0\ \stackrel{c)}{\longrightarrow}\ \alpha_0,
\]  
where a) the random operator $\alpha(\o)$ is represented through a variational inequality involving Fitzpatrick's representation $f(\o)$; b) the representation is ``scale integrated" to a $\o$-independent effective $f_0$; c) a maximal monotone operator $\alpha_0$ is associated to $f_0$. In Theorem \ref{newmain} we prove that $\alpha_0$ is the correct homogenization of $\alpha_\e$.   

In Section \ref{ssec:mmo} we review the properties of maximal monotone operators and their variational formulation due to Fitzpatrick. In Section \ref{ssec:stochan} we recall the basis of ergodic theory that we need in order to state our first main tool: Birkhoff's Ergodic Theorem. Section \ref{sec:stoch} is devoted to the translation to the stochastic setting of Visintin's scale integration-disintegration theory, which paves the way to our main result, Theorem \ref{newmain}. The applications we provide in the last section are: Ohmic electric conduction with Hall effect (Section \ref{sec:ohm}), and nonlinear elasticity, (Section \ref{ssec:nle}).

\section{Notation and preliminaries}

We use the notation $a\cdot b$ for the standard scalar product for vectors in $\R^n$. The arrows $\weak$ and $\weakst$ denote weak and   weak$^*$ convergence, respectively. As usual, $\mathscr D(D)$ stands for the space of $C^\infty$-functions with compact support in $D\subset \R^n$; its dual is denoted by $\mathscr D'(D)$.

\subsection{Maximal monotone operators}  
\label{ssec:mmo}

In this section we summarize the variational representation of maximal monotone operators introduced in \cite{Fitzpatrick}. Further details and proofs of the statements can be found, e.g., in \cite{VisintinCalcVar2013}. Let $B$ be a reflexive, separable and real Banach space; we denote by $B'$ its dual,  by $\mathcal{P}(B')$ the power set of $B'$, and by $\langle y,x \rangle$ the duality between $x\in B$ and $y \in B'$. Let $\alpha \colon B \rightarrow \mathcal{P}(B')$ be a set-valued operator and let 
\[
\mathcal{G}_\alpha := \{ (x,y) \in B\times B':y\in\alpha(x)\}
\]
be its graph. (We make use of the two equivalent notations $y \in \alpha(x)$ or $(x,y)\in \mathcal{G}_\alpha$.) The operator $\alpha$ is said to be {\it monotone} if 
\begin{equation}
\label{mon}
	(x,y) \in \mathcal{G}_\alpha\quad \Rightarrow\quad \langle y-y_0, x-x_0\rangle \ge 0, \quad \forall (x_0,y_0)\in \mathcal{G}_\alpha
\end{equation}
and {\it strictly monotone} if there is $\theta>0$ such that
\begin{equation}
\label{mon-strict}
	(x,y) \in \mathcal{G}_\alpha \quad\Rightarrow\quad \langle y-y_0, x-x_0\rangle \ge \theta \|x-x_0\|^2, 
		\quad  \forall (x_0,y_0)\in \mathcal{G}_\alpha.
\end{equation}
We denote by $\alpha^{-1}$ the inverse operator in the sense of graphs, that is 
\[
	x \in \alpha^{-1}(y)\quad \Leftrightarrow\quad y\in \alpha(x).
\]
The monotone operator $\alpha$ is said to be {\it maximal monotone} if the reverse implication of \eqref{mon} is fulfilled, namely if 
\[
	\langle y-y_0, x-x_0\rangle \ge 0\quad  \forall (x_0,y_0)\in \mathcal{G}_\alpha\quad  \Rightarrow\quad  (x,y) \in \mathcal{G}_\alpha.
\]
An operator $\alpha$ is maximal monotone if and only if $\alpha^{-1}$ is maximal monotone. 
For any operator $\alpha \colon B \rightarrow \mathcal{P}(B')$, which is not identically $\emptyset$, we introduce the {\it Fitzpatrick function of $\alpha$} as the function $f_\alpha \colon B \times B' \to \R$ defined by
\begin{align*}
	f_{\alpha}(x,y):&= \langle y, x \rangle + \sup\{\langle y-y_0,x_0-x \rangle :(x_0,y_0)\in\mathcal{G}_\alpha\}\\
			&=\sup\left\{\langle y,x_0 \rangle + \langle y_0,x \rangle - \langle y_0,x_0 \rangle   :(x_0,y_0)\in\mathcal{G}_\alpha\right\}.
\end{align*}
As a supremum of a family of linear functions, the Fitzpatrick function $f_\alpha$ is convex and lower semicontinuous. Moreover (see \cite{Fitzpatrick}) 
\begin{lemma}
\label{lemma:fp}
An operator $\alpha$ is monotone if and only if
\[
	(x,y) \in \mathcal{G}_\alpha\quad \Rightarrow\quad f_{\alpha}(x,y)=\langle y, x \rangle,
\]
while $\alpha$ is maximal monotone if and only if
\[
	\left\{
		\begin{array}{ll}
			f_{\alpha}(x,y) \ge \langle y, x \rangle\ \quad \forall (x,y)\in B\times B'\\
			f_{\alpha}(x,y)=\langle y, x \rangle\quad \iff (x,y) \in \mathcal{G}_\alpha.
		\end{array}
	\right.
\]
\end{lemma}
In the case $B=B'=\R$, it is easy to compute some simple examples of Fitzpatrick function of a monotone operator:

1. Let $\alpha(x):=ax+b$, with $a>0,b\in \R$. A straightforward computation shows that 
\[
	f_\alpha(x,y) = \frac{(y-b+ax)^2}{4a} +bx.
\]

2. Let 
\[
	\alpha(x) = \left\{
		\begin{array}{cl}
			1 & \text{if }x > 0, \\
			{[0,1]} & \text{if }x=0, \\
			-1 & \text{if }x<0.
		\end{array}
		\right.
\]			
Then
\[
	f_\alpha(x,y)=\left\{
		\begin{array}{cl}
			|x| & \text{if } |y| \leq 1, \\
			+\infty & \text{if }|y|>1.
		\end{array}
		\right.
\]
and in both cases $f_\alpha$ coincides with $g(x,y)=xy$ exactly on the graph of $\alpha$. 

We define $\mathcal{F}=\mathcal{F}(B)$ to be the class of all proper, convex and lower semicontinuous functions $f \colon B\times B'\rightarrow \mathbb{R}\cup \{+\infty \}$ such that
\[
	f(x,y)\ge \langle y, x \rangle \quad \forall (x,y)\in B\times B'.
\]
We call $\mathcal F(B)$ the class of \emph{representative functions}. The above discussion shows that given a monotone operator $\alpha$, one can construct its representative function in $\mathcal F(B)$, and viceversa, given a function $f\in \mathcal F(B)$, we define the operator represented by $f$, which we denote $\alpha_f$, by: 
\begin{equation}
\label{def:graph}
	(x,y) \in \mathcal G_{\alpha_f} \Leftrightarrow f(x,y)=\langle y,x \rangle.
\end{equation}	
 A crucial point is whether $\alpha_f$ is monotone (or maximal monotone, see also \cite[Theorem 2.3]{VisintinAsAn2013}). 
\begin{lemma}
\label{lemma:antifp}
	Let $f\in \mathcal F(B)$, then 
	\begin{itemize}
		\item[(i)] the operator $\alpha_f$ defined by \eqref{def:graph} is monotone;
		\item[(ii)] the class of maximal monotone operators is strictly contained in the class of operators representable by functions in $\mathcal F(B)$.
	\end{itemize}	
\end{lemma}
\begin{proof}
(i) If $\mathcal G_{\alpha_f}$ is empty or reduced to a single element, then the statement is trivially satisfied. Let $x_1,x_2 \in B$, $y_i \in \alpha_f(x_i)$, and assume, by contradiction, that $\langle y_2 - y_1, x_2 -x_1 \rangle <0.$ Define $P_i:=(x_i,y_i)\in B \times B'$ and $g(x,y):= \langle y,x \rangle$. We compute
\begin{align*}
	g\left( \frac{P_1+P_2}{2}\right) - \frac{g(P_1) +g(P_2)}{2} &= 
		\frac14 \big( \langle y_1 +y_2,x_1+x_2\rangle \big) - \frac12 \big(  \langle y_1 ,x_1\rangle + \langle y_2,x_2\rangle \big) \\
		&= \frac14 \big( \langle y_1 ,x_2\rangle + \langle y_2 ,x_1\rangle - \langle y_1 ,x_1\rangle -\langle y_2 ,x_2\rangle\big) \\
		& = -\frac14 \big( \langle y_2 -y_1 ,x_2 - x_1\rangle \big) >0.
\end{align*}
Since $f \geq g$ and $f(x_i,y_i)=g(x_i,y_i)$, the last inequality implies
\[
		f\left( \frac{P_1+P_2}{2}\right) > \frac{f(P_1) +f(P_2)}{2},
\]
which contradicts the convexity of $f$. 

(ii) Maximal monotone operators are representable by Lemma \ref{lemma:fp}. To see that the inclusion is strict, assume that $\alpha_f$ is maximal monotone. Let $(x_0,y_0)\in \mathcal G_{\alpha_f}$ and fix $c\in \R$ such that $c > f(x_0,y_0)$. The function
\[
	h(x,y) = \max\{c,f(x,y)\}
\]
clearly belongs to $\mathcal F(B)$, but $\mathcal G_{\alpha_h}$ is strictly contained in $\mathcal G_{\alpha_f}$, since 
\[
	h(x_0,y_0) \geq c >f(x_0,y_0)=\langle y_0,x_0 \rangle,
\]	 
and thus $\alpha_h$ is not maximal.
\end{proof}

\begin{oss} When $\alpha = \partial \varphi$ for some proper, convex, lower semicontinuous function $\varphi:B \to \R \cup \{+\infty\}$, the classical Fenchel equality yields 
\[
	\varphi(x) + \varphi^*(y) \geq \langle y,x \rangle\quad \forall\,(x,y)\in B \times B',
\]	
\[	
	y \in \alpha(x)\quad \Leftrightarrow \quad \varphi(x) + \varphi^*(y) = \langle y,x \rangle.
\]
Thus, Fitzpatrick's representative function $f_\alpha$ generalizes $\varphi + \varphi^*$ to maximal monotone operators which are not subdifferentials. Remark that, even if $\alpha=\partial \varphi$, in general $f_\alpha \neq \varphi + \varphi^*$: for example, let $\alpha(x):=x$ in $\R$, then 
\[
	f_\alpha(x,y) = \frac{(x+y)^2}{4} \neq \frac{x^2}{2}+\frac{y^2}{2} = \varphi(x)+\varphi^*(y).
\]
\end{oss}
\bigskip

We need to introduce also parameter-dependent operators. 
For any measurable space $X$ we say that a set-valued mapping $g\colon X \rightarrow \mathcal P(B')$ is measurable if for any open set $R\subseteq B'$ the set 
\[
	g^{-1}(R) := \{ x \in X : g(x) \cap R \neq \emptyset \}
\]
is measurable. 

Let $\mathcal B(B)$ be the $\sigma$-algebra of the Borel subsets of the separable and reflexive Banach space $B$, let $(\Omega,\mathcal{A},\mu)$ be a probability space equipped with the $\sigma$-algebra $\mathcal{A}$ and the probability measure $\mu$. We define a random (maximal) monotone operator as $\alpha \colon B \times \Omega \rightarrow \mathcal{P}(B')$ such that
\begin{eqnarray}
	&\alpha \textrm{ is $\mathcal B(B)\otimes \mathcal{A}$-measurable}, \label{1} \\
	&\alpha(x,\omega) \textrm{ is closed for any $x \in B$ and for a.e. }\omega \in \Omega, \label{2} \\
	&\alpha(\cdot,\omega) \textrm{  is (maximal) monotone  for  a.e. }\omega \in \Omega.\label{3}
\end{eqnarray}
If $\alpha$ fulfills \eqref{1} and \eqref{2} then for any $\mathcal{A}$-measurable mapping $v:\Omega \rightarrow B$, the multivalued mapping $\omega \mapsto \alpha(v(\omega),\omega)$ is closed-valued and measurable. We denote by $\mathcal{F}(\Omega;B)$ the set of all {\it measurable representative functions} $ f: B \times B' \times \Omega \to \R \cup \{+\infty\}$ such that 
\begin{itemize}
	\item[\rm(a)] $f\colon B\times B' \times \Omega \rightarrow \mathbb{R} \cup \{+\infty\}$ 
		is $\mathcal{B}(B \times B') \otimes \mathcal{A}$ -measurable,
	\item[\rm(b)] $f(\cdot,\cdot,\omega)$ is convex and  lower semicontinuous for a.e.\,$\omega \in \Omega$,
	\item[\rm(c)] $f(x,y,\omega)\ge \langle y,x \rangle$ for all $(x,y) \in B\times B'$ and for  a.e.\,$\omega \in \Omega$.
\end{itemize}
As above, $f \in \mathcal{F}(\Omega;B)$ represents the operator $\alpha=\alpha_f(\cdot,\omega)$ in the following sense:
\begin{equation}
\label{eq:represent}
	y \in \alpha(x,\omega)\ \Leftrightarrow \ f(x,y,\omega)=  \langle y,x\rangle \qquad \forall (x,y) \in B\times B',\, \text{for a.e. }\omega \in \Omega.
\end{equation}
Precisely, any measurable representative function $f\colon B \times B' \times\Omega \rightarrow \mathbb{R}\cup \{+\infty \}$ represents a closed-valued, measurable, monotone operator  $\alpha \colon B \times \Omega \rightarrow\mathcal{P}(B')$, while any closed-valued, measurable, maximal monotone operator $\alpha$ can be represented by a measurable representative function $f$, for instance, by its Fitzpatrick function \cite[Proposition 3.2]{VisintinAsAn2013}.

\subsection{Stochastic analysis}
\label{ssec:stochan}

In this subsection we review the basic notions and results of stochastic analysis that we need in Section \ref{sec:stoch}. For more details see \cite[Chapter 7]{JikovKozlovOleinik94}. Let $(\Omega, \mathcal{A}, \mu)$ be a probability space, where $\mathcal{A}$ is a $\sigma$-algebra of subsets of $\Omega$ and $\mu$ is a probability measure on $\Omega$. Let $n \in \N$ with $n \geq 1$. A {\it $n$-dimensional dynamical system} $T$ on $\Omega$ is a family of mappings $T_x \colon \Omega \to \Omega$, with $x \in \R^n$, such that
\begin{itemize}
\item[(a)] $T_0$ is the identity and $T_{x+y}=T_xT_y$ for any $x,y \in \R^n$;
\item[(b)] for every $x \in \R^n$ and every set $E \in \mathcal{A}$ we have $T_xE \in \mathcal{A}$ and
\begin{equation}\label{invariancef}\mu(T_xE)=\mu(E)\end{equation}
\item[(c)] for any measurable function $f \colon \Omega \to \R^m$, the function $\tilde f \colon \R^n \times \Omega \to \R^m$ given by
\[
\tilde f(x,\omega)=f(T_x\omega)
\]
is measurable.
\end{itemize}
Given a $n$-dimensional dynamical system $T$ on $\Omega$, a measurable function
$f$ defined on $\Omega$ is said to be {\it invariant} if $f(T_x\omega)=f(\omega)$ a.e.\,in $\Omega$, for each $x \in \R^n$. A dynamical system is said to be {\it ergodic} if the only invariant functions are the constants. The {\it expected value} of a random variable $f \colon \Omega \to \R^n$ is defined as
\[
{\mathbb E}(f):=\int_\Omega f\,d\mu.
\]

In the context of stochastic homogenization, it is useful to provide an orthogonal decomposition of $L^2(\O)$ into functions, the realizations of which are vortex-free and divergence-free, in the sense of distributions (see, e.g., \cite[Section 7]{JikovKozlovOleinik94}). For $p \in [1,+\infty[$, Peter-Weyl's decomposition theorem \cite{PeterWeyl27} can be generalized to a relation of orthogonality between subspaces of $L^p(\O)$ and $L^{p'}(\O)$.
 Let $v\in L^p_{\rm loc}(\R^n;\R^n)$. We say that $v$ is {\it potential} (or {\it vortex-free}) in $\mathbb{R}^{n}$ if
\[
\int \bigg( v^{i}\frac{\partial\varphi}{\partial x_{j}}-v^{j}\frac{\partial \varphi}{\partial x_{i}}\bigg)\,dx=0, \ \ \ \ \forall i,j=1,\dots,n,\, \forall \varphi \in \mathscr{D}(\mathbb{R}^{n})
\]
and we say that $v$ is {\it solenoidal} (or {\it divergence-free}) in $\mathbb{R}^{n}$ if
\[
	\sum_{i=1}^{n}\int v^{i}\frac{\partial\varphi}{\partial x_{i}}\,dx=0, \ \ \ \forall \varphi \in \mathscr{D}(\mathbb{R}^{n}).
\]
Next we consider a vector field on $(\Omega,\mathcal{A},\mu)$. We say that $f\in L^{p}(\Omega;\R^n)$ is {\it potential} if $\mu$-almost all its realizations $x\mapsto f(T_{x}\omega)$ are potential. We denote by $L_{\rm pot}^{p}(\Omega;\R^n)$ the space of all potential $f\in L^{p}(\Omega;\R^n)$. In the same way, $f\in L^{p}(\Omega;\R^n)$ is said to be {\it solenoidal} if $\mu$-almost all its realizations $x\mapsto f(T_{x}\omega)$ are solenoidal and we denote by $L_{\rm sol}^{p}(\Omega;\R^n)$ the space of all solenoidal $f\in L^{p}(\Omega;\R^n)$. In the following Lemma we collect the main properties of potential and solenoidal $L^p$ spaces (see \cite[Section 15]{JikovKozlovOleinik94}).
\begin{lemma}
\label{lemma:solpot} 
Define the spaces
\begin{align*}
	\mathcal{V}^{p}_{\rm pot}(\Omega;\R^n)&:=\{ f\in L_{\rm pot}^{p}(\Omega;\R^n) : \mathbb E(f)=0\}, \\
	\mathcal{V}^{p}_{\rm sol}(\Omega;\R^n)&:=\{ f\in L_{\rm sol}^{p}(\Omega;\R^n) : \mathbb E(f)=0\}.
\end{align*}
The spaces $L_{\rm pot}^{p}(\Omega;\R^n), L_{\rm sol}^{p}(\Omega;\R^n), \mathcal{V}^{p}_{\rm pot}(\Omega;\R^n), \mathcal{V}^{p}_{\rm sol}(\Omega;\R^n)$ are closed subspaces of $L^{p}(\Omega;\R^n)$. If $u \in L_{\rm sol}^p(\Omega;\R^n)$ and $v \in L_{\rm pot}^q(\Omega;\R^n)$ with $1/p+1/q=1$, then 
\begin{equation}
\label{ort}
	\mathbb E(u \cdot v)=\mathbb E(u) \cdot \mathbb E(v)
\end{equation}
and the relations
\[
	(\mathcal{V}^{p}_{\rm sol}(\Omega;\R^n))^\perp = \mathcal{V}^{q}_{\rm pot}(\Omega;\R^n) \oplus \R^n,\qquad (\mathcal{V}^{p}_{\rm pot}(\Omega;\R^n))^\perp = \mathcal{V}^{q}_{\rm sol}(\Omega;\R^n) \oplus \R^n
\]
hold in the sense of duality between the spaces $L^p(\O)$ and $L^q(\Omega)$.	
\end{lemma}

One of the most important results regarding stochastic homogenization is Birkhoff's Ergodic Theorem. We report the statement given in \cite[Theorem 7.2]{JikovKozlovOleinik94}.
\begin{theorem}{\bf(Birkhoff's Ergodic Theorem)}
\label{B}
Let $f \in L^1(\Omega;\R^m)$ and let $T$ be a $n$-dimensional ergodic dynamical system on $\Omega$. Then
\[
	\mathbb E(f)=\lim_{\e \to 0}\frac{1}{|K|}\int_K f\big(T_{x/\e}\omega\big)\,dx 
\]
for $\mu$-a.e. $\omega \in \Omega$, for any $K \subset \R^n$ bounded, measurable, with $|K|>0$.
\end{theorem}
\begin{oss}
Birkhoff's theorem implies that $\mu$-almost every realization $\tilde f_\e(x)=f(T_{x/\e}\omega)$ satisfies
\[ 
	\lim_{\e \to 0} \frac{1}{|K|}\int_K \tilde f_\e(x)\, dx = \mathbb{E}(f).
\]	
Since this holds for every measurable bounded set $K \subset \R^n$, it entails in particular that if $f\in L^p(\Omega)$, then
\begin{equation}
\label{birk:weak} 
	\tilde f_\e \weak \mathbb{E}(f) \quad \mbox{weakly in }L^p_{\rm loc}(\R^n;\R^m).
\end{equation}	
\end{oss}

In what follows, the dynamical system $T_x$ is assumed to be ergodic and $K\subset \R^n$ is bounded, measurable and $|K|>0$.

\section{Stochastic homogenization}
\label{sec:stoch}

Let be given a probability space $(\Omega,\mathcal{A},\mu)$ endowed with a $n$-dimensional ergodic dynamical system $T_x \colon \Omega \to \Omega$, $x\in \R^n$. Let $p\in (1,+\infty)$, $q={\frac{p}{p-1}}$ and let $\alpha$ be a random maximal monotone operator, as in \eqref{1}--\eqref{3}. 

\subsection{Stochastic scale integration/disintegration}
\label{ssec:sid}

We translate here Visintin's  scale integration/disintegration \cite{Visintin09, VisintinAsAn2013} to the stochastic homogenization setting.
\begin{oss} While most of this subsection's statements are Visintin's results written in a different notation, some others contain a small, but original contribution. Namely: Lemma \ref{lemmaV} can be found in \cite[Lemma 4.1]{VisintinAsAn2013}, where the assumption of boundedness for $K$ is used to obtain the lower semicontinuity of the inf function. Since we prefer not to impose this condition, we independently proved the lower semicontinuity part, making use of the coercivity of $g$. 
Lemma \ref{lemma:estima} and Proposition \ref{prop-f0} were given for granted in \cite{VisintinAsAn2013}, but we decided to write a proof for sake of clarity.  
Lemma \ref{lemma:problema} and Lemma \ref{lemma:scaleintegration}, are essentially \cite[Theorem 4.3]{VisintinAsAn2013} and \cite[Theorem 4.4]{VisintinAsAn2013}, cast in the framework of stochastic homogenization in the probability space $(\Omega,\mathcal{A},\mu)$, instead of periodic homogenization on the $n$-dimensional torus. Theorem \ref{th:inherit} collects other results of \cite{VisintinAsAn2013}. Lemma \ref{lemma:strictmon} is an original remark.
\end{oss}
Let $f(\cdot,\cdot,\omega) \colon \R^n \times \R^n \to \R \cup \{+\infty\}$ be the Fitzpatrick representation of the operator $\alpha(\cdot,\o)$. We assume the following coercivity condition on $f$: there exist $c>0$ and $k\in L^1(\Omega)$ such that for any $\xi,\eta\in \R^n$, for any $\omega \in \Omega$ it holds 
\begin{equation}\label{coercivity}
	f(\xi,\eta,\omega) \ge c\left(|\xi |^p+|\eta |^q \right)+k(\omega).
\end{equation}
We define the homogenized representation $f_0 \colon \R^n \times \R^n \to \R  \cup \{+\infty\}$ as 
\begin{equation}\label{f0}
	f_0(\xi,\eta):=\inf \bigg\{ \int_\Omega f(\xi+v(\omega),\eta+u(\omega),\omega) \,d\mu :u \in \mathcal V^p_{\rm pot}(\O;\R^n), v\in \mathcal V^q_{\rm sol}(\O;\R^n) \bigg\}.
\end{equation}

\begin{lemma}
\label{lemmaV}
Let $X$ be a reflexive Banach space, let $K$ be a weakly closed subset of a reflexive Banach space. Let the function $g \colon X \times K \to \R \cup \{+\infty\}$ be weakly lower semicontinuous and bounded from below. If $g$ is coercive, e.g.\,in the sense that for all $M>0$ the set
\[
	\{(x,y)\in X \times K : g(x,y) \leq M\} 
\]
is bounded, then the function $h\colon X \to \R \cup \{+\infty\}$ given by 
\[
	h(x):= \inf_{y\in K}g(x,y)
\]
is weakly lower semicontinuous and coercive. Moreover, if $g$ and $K$ are convex then $h$ is convex.
\end{lemma}

\begin{proof}
	Let $x_j \rightharpoonup x \in X$, we must show that 
\begin{equation}
\label{eq2}
		\liminf_{j \to +\infty} h(x_j) \geq h(x).
\end{equation}
Let 
\[
	\ell:=\liminf_{j \to +\infty} h(x_j).
\]	 
If $\ell = +\infty$, then \eqref{eq2} is trivially satisfied. On the other hand, since $g$ is bounded from below, then $\ell > -\infty$, and we can assume that $\ell \in \R$. By definition of inferior limit, there exists a subsequence of $\{x_j\}$ (not relabeled), such that $\lim_{j\to \infty}h(x_j)=\ell$. Up to extracting another subsequence, we can also assume that $h(x_j)\leq 2\ell$ for all $j\in \mathbb N$. Let $\e >0$ be fixed, by definition of infimum, for all $j\in \mathbb N$, there exists $y_j \in K$ such that
\begin{equation}
\label{eq3}
	h(x_j) = \inf_{y \in K} g(x_j,y) \geq g(x_j,y_j)-\e.
\end{equation}
Therefore
\[
	g(x_j,y_j) \leq 2\ell +\e\qquad \forall\,j\in \mathbb N.
\]
By the coercivity assumption on $g$, we deduce that $y_j$ is bounded, we can therefore extract a subsequence $\{y_{j_k}\} \subset K$ such that $y_{j_k}\rightharpoonup y$. Since $K$ is weakly closed, then $y \in K$. We can now pass to the inferior limit in \eqref{eq3}, using the lower semicontinuity of $g$
\begin{equation}
\label{eq:4}
	 \liminf_{k \to +\infty} h(x_{j_k})  \geq  \liminf_{k \to +\infty} g(x_{j_k},y_{j_k})-\e \geq g(x,y) -\e \geq h(x)-\e.
\end{equation}
By arbitrariness of $\e>0$, this proves the weak lower semicontinuity of $h$. Assume now that $K$ is convex. Take $\lambda\in [0,1]$, $x_1,x_2\in X$ and  $y_1,y_2 \in K$. By convexity of $g$
\[
	h(\lambda x_1+(1-\lambda)x_2)\le g(\lambda x_1+(1-\lambda)x_2,\lambda y_1+(1-\lambda)y_2)\le \lambda g(x_1,y_1) +(1-\lambda)g(x_2,y_2).
\]
Passing to the infimum with respect to $y_1,y_2\in K$ we conclude
\[
	h(\lambda x_1+(1-\lambda)x_2)\le \lambda h(x_1) +(1-\lambda)h(x_2).
\]
Regarding the coercivity of $h$, denote
\[
	 B_t:=\{x \in X : h(x) \leq t\},\qquad A_t:=\{x\in X : g(x,y)\leq t, \text{ for some }y\in K\}.
\]
Let $M,\e>0$, for all $x\in B_M$ there exists $y\in K$ such that $g(x,y)\leq h(x)+\e\leq M+\e$, therefore $ B_M \subseteq A_{M+\e}$. Since $g$ is coercive, $A_{M+\e}$ is bounded and thus $B_M$ is bounded, i.e., $h$ is coercive.
\end{proof}
In the proof of Proposition \ref{prop-f0} we need the following estimate
\begin{lemma}
\label{lemma:estima}
	For all $p\in[1,+\infty[$ there exists $C>0$ such that
	\[
		\int_\Omega |\xi + u(\omega)|^p\,d\mu \geq C \int_\Omega |\xi|^p + |u(\omega)|^p\,d\mu 
	\]
	for all $\xi \in \R^n$, for all $u\in L^p(\Omega;\R^n)$ such that $\mathbb E(u)=0$.
\end{lemma}
\begin{proof}
	Consider the operator
	\begin{align*}
		\Phi &:L^p(\Omega;\R^n) \to L^p(\Omega;\R^n) \times L^p(\Omega;\R^n) \\
			&\qquad u\quad \mapsto \left(\mathbb  E(u), u -\mathbb  E(u)\right).
	\end{align*}		 
On the image space $L^p(\Omega;\R^n) \times L^p(\Omega;\R^n)$, choose the equivalent norm
\[
	\|(u,v)\|_{L^p \times L^p}:=\|u\|_{L^p} + \|v\|_{L^p}.
\]
Clearly, $\Phi$ is linear and continuous. Therefore, there exists $C>0$ such that
\[
	\int_\Omega |\mathbb E(u)|^p d\mu +\int_\Omega |u(\omega) - \mathbb E(u)|^p\,d\mu \leq \|\Phi(u)\|^p_{L^p \times L^p} \leq C\|u\|^p_{L^p}=C\int_\Omega |u(\omega)|^p\,d\mu.
\]	
Apply now the last inequality to $u(\omega)=\xi + \tilde u(\omega)$, with $\mathbb E(\tilde u)=0$:
\[
	 \int_\Omega |\xi|^p +  |\tilde u(\omega)|^p\,d\mu \leq C\int_\Omega |\xi +\tilde u(\omega)|^p\,d\mu.
\]
\end{proof}

\begin{proposition}
\label{prop-f0}
For all $(\xi,\eta)\in \R^n \times \R^n$ there exists a couple $(\tu,\tv)\in \mathcal V^p_{\rm sol}(\O;\R^n) \times \mathcal V^q_{\rm pot}(\O;\R^n)$ such that the infimum on the right-hand side of \eqref{f0} is attained. Moreover, $f_0 \in \mathcal F(\R^n)$. In particular, it holds
\begin{equation}
\label{ineq:f0}
	f_0(\xi,\eta)\ge \xi\cdot \eta\quad  \forall (\xi,\eta)\in \R^n\times \R^n.
\end{equation}
\end{proposition}

\begin{proof}
Let $K:=\mathcal V^p_{\rm sol}(\O;\R^n)\times \mathcal V^q_{\rm pot}(\O;\R^n)$. Then $K$ is weakly closed in $L^p(\O;\R^n) \times L^q(\O;\R^n)$ since it is closed and convex. Let $\xi,\eta\in \R^n$ be fixed, for any $(u,v) \in K$ let 
\[
	F_{\xi,\eta}(u,v):= \int_\Omega f(\xi+v(\omega),\eta+u(\omega),\omega) \,d\mu.
\]
We prove that the problem $\inf_KF_{\xi,\eta}$ has a solution applying the direct method of the Calculus of Variations. First, by \eqref{coercivity}, $\inf_KF_{\xi,\eta}>-\infty$. Then, if $(u_h,v_h)\in K$ is a minimizing sequence for $F_{\xi,\eta}$, by the coercivity assumption \eqref{coercivity}, up to subsequences, $(u_h,v_h) \rightharpoonup (u,v)$ weakly in $L^p(\O;\R^n) \times L^q(\O;\R^n)$, therefore $(u,v)\in K$, since $K$ is weakly closed. Finally, $F_{\xi,\eta}$ is $L^p\times L^q$-weakly lower semicontinuous since $f(\cdot,\cdot,\omega)$ is convex, lower semicontinuous, and bounded from below by an integrable function \eqref{coercivity}, therefore 
\[
	F_{\xi,\eta}(u,v)\leq \liminf_{h \to \infty}F_{\xi,\eta}(u_h,v_h) = \inf_K F_{\xi,\eta}.
\] 
This concludes the first part of the statement. We now want to show that $f_0\in \mathcal F(\R^n)$. Owing to \eqref{coercivity} and Lemma \ref{lemma:estima}, for all $\xi,\eta,(u,v)\in \R^n\times \R^n \times K$, there exists a constant $C>0$ such that 
\begin{align*}
	F_{\xi,\eta}(u,v) & \geq c\int_\Omega |\xi +v(\omega)|^p + |\eta +u(\omega)|^q +k(\omega)\,d\mu \\
		& 	\geq C\int_\Omega |\xi|^p +|u(\omega)|^p + |\eta|^q +|v(\omega)|^q d\mu+\mathbb E(k) \\
		&\geq C\left( |\xi|^p + {\|u\|}^p_{L^p(\Omega)} + |\eta|^q + {\|v\|}^q_{L^q(\Omega)} \right) -{\|k\|}_{L^1(\O)}.
\end{align*}
Thus, for any $M\ge 0$, the set 
\[
	\left\{(\xi,\eta,(u,v))\in {\R}^n\times \R^n \times K : F_{\xi,\eta}(u,v) \leq M\right\} 
\]
is bounded in $\R^n \times \R^n \times L^p(\O;\R^n) \times L^q(\O;\R^n)$. We are therefore in a position to apply Lemma \ref{lemmaV} and to conclude that $f_0$ is convex and lower semicontinuous. Furthermore, let $(\tu,\tv)\in K$ be a minimizer of $F_{\xi,\eta}$, using \eqref{ort} 
\[
\begin{aligned}
	f_0(\xi,\eta)&= \int_\Omega f(\xi+\tu(\omega),\eta+\tv(\omega),\omega) \,d\mu\\
	&\ge  \int_\Omega (\xi+\tu(\omega))\cdot (\eta+\tv(\omega)) \,d\mu\\
	& =\mathbb E(\xi+\tu) \cdot \mathbb E(\eta + \tv)\\
	& =\xi\cdot\eta,
\end{aligned}
\]
which yields the conclusion.
\end{proof}

We denote by $\alpha_0$ the operator on $\R^n$ represented by $f_0$ through the usual relation
\[
	\eta \in \alpha_0(\xi) \quad \Leftrightarrow\quad f_0(\xi,\eta)= \xi \cdot \eta.
\] 
We refer to $\alpha_0$ as the \textit{scale integration} of $\alpha$, since it is obtained through $f_0$, which is the scale integration of the Fitzpatrick representation $f$ of $\alpha$. 

\begin{lemma}
\label{lemma:problema}
Let $\xi,\eta\in \R^n$ be such that $\eta \in \alpha_0(\xi)$. Then, there exist $u\in L^p_{\rm sol}(\O;\R^n)$ and $v\in L^q_{\rm pot}(\O;\R^n)$ such that
\begin{equation}
\label{auxiliary}
 	v(\omega)\in \alpha(u(\omega),\o), \quad \text{for a.e. }\o\in\O.
\end{equation}
Moreover, $\mathbb{E}(u)=\xi$ and $\mathbb{E}(v)=\eta$, that is 
\begin{equation}\label{step2}
	\mathbb E( v) \in \alpha_0(\mathbb E( u)).
\end{equation}
\end{lemma}

\begin{proof}
Since $f_0$ represents $\alpha_0$, $\eta \in \alpha_0(\xi)$ implies
\begin{equation}
\label{eq:f0repr}
	f_0(\xi,\eta) = \xi \cdot \eta.
\end{equation}
Take now $\tu \in \mathcal V^p_{\rm sol}(\O;\R^n)$ and $ \tv \in \mathcal V^q_{\rm pot}(\O;\R^n)$ such that
\begin{equation}
\label{eq:hyp1}
		f_0(\xi,\eta)= \int_\Omega f(\xi+\tu(\omega),\eta+\tv(\omega),\omega) \,d\mu.
\end{equation}
Since $\mathbb E(\tu)=\mathbb E(\tv)=0$,
\begin{align*}
\xi \cdot \eta&=\mathbb E(\xi+\tu) \cdot \mathbb E(\eta+\tv)\\
	& \stackrel{\eqref{ort}}{=}\int_\O (\xi+\tu(\o))\cdot (\eta+\tv(\o))\,d\mu\\
	&\stackrel{f \in \mathcal F(\R^n)}{\le} \int_\Omega f(\xi+\tu(\omega),\eta+\tv(\omega),\omega) \,d\mu\\
	& \stackrel{\eqref{eq:hyp1}}{=} f_0(\xi,\eta)\\
	& \stackrel{\eqref{eq:f0repr}}{=}\xi \cdot \eta
\end{align*}
from which we obtain 
\begin{equation}
\label{eq:312}
	(\xi+\tu(\o))\cdot (\eta+\tv(\o))= f(\xi+\tu(\omega),\eta+\tv(\omega),\omega), \quad \text{a.e. }\o\in\O.
\end{equation}
 Let $u(\omega):=\xi+\tu(\omega)$ and $ v(\omega):=\eta+\tv(\omega)$. Since $f$ represents $\alpha$, \eqref{eq:312} is equivalent to \eqref{auxiliary}. 
Moreover, since $\mathbb E(u)=\xi$ and $\mathbb E( v)=\eta$, $\eta \in \alpha_0(\xi)$ implies also \eqref{step2}.
\end{proof}

Lemma \ref{lemma:problema} is also referred to as \textit{scale disintegration} (see \cite[Theorem 4.4]{VisintinAsAn2013}), as it shows that given a solution $(\xi,\eta)$ to the integrated problem $\eta \in \alpha_0(\xi)$, it is possible to build a solution to the original problem $v(\o) \in \alpha(u(\o),\o)$. The converse, known as \textit{scale integration} (see \cite[Theorem 4.3]{VisintinAsAn2013}) is provided by the next Lemma.

\begin{lemma}
\label{lemma:scaleintegration}
Let $u\in L^p_{\rm sol}(\O;\R^n)$ and $v\in L^q_{\rm pot}(\O;\R^n)$ satisfy
\begin{equation}
\label{auxiliary:int}
	 v(\omega)\in \alpha(u(\omega),\o), \quad \textrm{for a.e. }\o\in\O,
\end{equation}
then
\begin{equation}
\label{eq:integrated}
	\mathbb E(v) \in \alpha_0(\mathbb E(u)).
\end{equation}
\end{lemma}
\begin{proof}
By \eqref{auxiliary:int} and \eqref{ort}
\[
	\int_\Omega f(u(\o),v(\o),\o)\,d\mu =  \int_\Omega u(\o)\cdot v(\o)\,d\mu = \mathbb E(u)\cdot \mathbb E(v).
\]
On the other hand, by definition of $f_0$,
\[
	\int_\Omega f(u(\o),v(\o),\o)\,d\mu \geq f_0(\mathbb E(u), \mathbb E(v)) \geq \mathbb E(u)\cdot \mathbb E(v).
\]	
We conclude that $f_0(\mathbb E(u), \mathbb E(v)) = \mathbb E(u)\cdot \mathbb E(v)$, which yields \eqref{eq:integrated}.
\end{proof}
How the properties of $\alpha$ and $f$ reflect on $\alpha_0$ and $f_0$ was thoroughly studied in \cite{VisintinAsAn2013}:
\begin{theorem}
\label{th:inherit}
If 
 \begin{itemize}
 	\item $f\in \mathcal F(\O,\R^n)$ is uniformly bounded from below, 
	\item there exists $(u,v)\in L^p(\O;\R^n) \times L^q(\O;\R^n)$ such that
	\[
		\int_\O f(u(\o),v(\o),\o)\,d\mu <+\infty,
	\]
	\item $f$ represents a maximal monotone operator for $\mu$-a.e. $\o \in \O$
\end{itemize}	
then $f_0$ represents a (proper) maximal monotone operator  \textup{\cite[Theorem 5.3]{VisintinAsAn2013}}. Moreover, if $f$ is strictly convex, then 
\begin{itemize}
	\item $f_0$ is strictly convex \textup{\cite[Lemma 5.4]{VisintinAsAn2013}},
	\item the operators $\alpha_0$ and $\alpha_0^{-1}$ are both strictly monotone \textup{\cite[Proposition 5.5]{VisintinAsAn2013}} 
\end{itemize}
and if Dom$(\alpha_0)$ and Dom$(\alpha_0^{-1})$ are unbounded, then $\alpha_0$ and $\alpha_0^{-1}$ are coercive \textup{\cite[Proposition 5.6]{VisintinAsAn2013}}. 
\end{theorem}
In order to obtain strict monotonicity of $\alpha_0$ and $\alpha_0^{-1}$, by the next Lemma we provide an alternative to strict convexity of the Fitzpatrick function.
\begin{lemma}
\label{lemma:strictmon}
Let $\alpha(\cdot,\omega):B \to B'$ be maximal and strictly monotone, and assume that its Fitzpatrick representation $f$ is coercive, in the sense of \eqref{coercivity}. Then its scale integration $\alpha_0$ is strictly monotone.
\end{lemma}
\begin{proof}
	For all $\eta_i\in \alpha_0(\xi_i)$, $i=1,2$, by Lemma \ref{lemma:problema}, there exist $u_i\in L^p_{\rm sol}(\O;\R^n)$ and $v_i\in L^q_{\rm pot}(\O;\R^n)$ such that
\begin{equation}
\label{smon1}
 	v_i(\omega)\in \alpha(u_i(\omega),\o), \quad \text{for a.e. }\o\in\O
\end{equation}
and $\mathbb{E}(u_i)=\xi_i$,  $\mathbb{E}(v_i)=\eta_i$. By \eqref{ort}, strict monotonicity of $\alpha$, and Jensen's inequality
\begin{align*}
	(\eta_2-\eta_1)\cdot(\xi_2-\xi_1) &= \int_\O (v_2(\o)-v_1(\o))\cdot(u_2(\o)-u_1(\o)) d\mu \\
		& \geq \theta \int_\O |u_2(\o)-u_1(\o)|^2d\mu \\
		& \geq \theta \left|\int_\O u_2(\o)-u_1(\o)d\mu \right|^2\\
		& =\theta \left| \xi_2 -\xi_1 \right|^2.
\end{align*}
\end{proof}

\subsection{Main result}
Let $D\subset \R^n$ be a Lipschitz and bounded domain with $|D|>0$. 
We recall the following classical result.
\begin{lemma}[Div-Curl lemma, \cite{Murat78}]
\label{lemma:divcurl}
Let $p\in ]1,+\infty[$, $q:=p/(p-1)$. Let $v^n,v \in L^{q}(D)^m$ and $u^n,u \in L^p(D)^m$ be  such that
\[ 
	v_n \weak v \quad \mbox{weakly in }L^{q}(D)^m, \qquad u^n \weak u\quad \mbox{weakly in }L^p(D)^m.
\]	
In addition, assume that
\[
	\{ \Curl v^n\} \text{ is compact in }W^{-1,q}(D),\qquad
	\{\Div u^n \} \text{ is compact in }W^{-1,p}(D).
\]
Then	
\[ 
	v^n \cdot u^n \stackrel{*}{\weak} v \cdot u \qquad \mbox{in }\mathscr D'(D).
\]		
\end{lemma}

We are now ready to prove our main result concerning the stochastic homogenization of a maximal monotone relation.
\begin{theorem}
\label{newmain}
Let $(\Omega,\mathcal{A},\mu)$ be a probability space with an $n$-dimensional ergodic dynamical system $T_x:\Omega \to \Omega$, $x\in\R^n$. Let $D\subset \R^n$ be a bounded domain, let $p \in ]1,+\infty[$ and $q=p/(p-1)$. Let $\alpha: \R^n \times \Omega \to \mathcal P(\R^n)$ be a closed-valued, measurable, maximal monotone random operator, in the sense of \eqref{1}--\eqref{3}.

Let $f$ be a Fitzpatrick representation of $\alpha$, as in \eqref{eq:represent}. Let $f$ satisfy \eqref{coercivity} and assume that for $\mu$-a.e. $\o \in \Omega$ and $\e \geq 0$ there exists a couple 
\[
	(J_\o^\e,E_\o^\e)\in L^p(D;\R^n)\times L^q(D;\R^n)
\]	 
such that
\begin{subequations}
\begin{equation}
\label{e:bound}
	\{ \Div J_\o^\e\}_{\e \geq 0} \text{ is compact in }W^{-1,p}(D),\quad 	
		\{ \Curl E_\o^\e\}_{\e \geq 0} \text{ is compact in }W^{-1,q}(D;\R^3), 
\end{equation}   	
\begin{equation}	
\label{e:convergence}
	\lim_{\e \to 0} J_\o^\e = J_\o^0 \text{ weakly, in }L^p(D),\quad
		\lim_{\e \to 0} E_\o^\e = E_\o^0 \text{ weakly, in }L^q(D),
\end{equation}    
\begin{equation}
\label{e:inclusion}
	E_\o^\e(x) \in \alpha(J_\o^\e(x),T_{x/\e}\omega) \text{ a.e. in }D.
\end{equation}
\end{subequations}
Then, for $\mu$-a.e. $\o \in \Omega$
\begin{equation}	
\label{hom1}
		E_\o^0(x) \in \alpha_0(J_\o^0(x)) \text{ a.e. in }D,
\end{equation}
where $\alpha_0$ is the maximal monotone operator represented by the homogenized representation $f_0 \colon \R^n \times \R^n \to \R  \cup \{+\infty\}$ defined by
\[
	f_0(\xi,\eta):=\inf \bigg\{ \int_\Omega f(\xi+u(\omega),\eta+v(\omega),\omega) \,d\mu :u \in \mathcal V^p_{\rm sol}(\O;\R^n), v\in \mathcal V^q_{\rm pot}(\O;\R^n) \bigg\}.
\]

\end{theorem}
\begin{proof}
By Lemma \ref{lemma:problema} for all $\xi,\eta\in \R^n$ such that $\eta \in \alpha_0(\xi)$  there exist $u\in L^p_{\rm sol}(\O;\R^n)$ and $v\in L^q_{\rm pot}(\O;\R^n)$ such that $\mathbb{E}(u)=\xi$, $\mathbb{E}(v)=\eta$, and
\begin{equation}
\label{auxiliary2}
 	v(\omega)\in \alpha(u(\omega),\o), \quad \text{for a.e. }\o\in\O.
\end{equation}
Define the stationary random fields $u_\e, v_\e \colon D\times \O \to \R^n$ as
\[
	u_\e(x,\omega):=u(T_{x/\e}\omega), \quad v_\e(x,\omega):=v(T_{x/\e}\omega).
\]
By \eqref{birk:weak}, for a.e. $\o \in \O$
\[
	x \mapsto u_\e(x,\omega) \in L^p_{\rm loc}(\R^n;\R^n),\qquad x \mapsto v_\e(x,\omega) \in L^q_{\rm loc}(\R^n;\R^n).
\]
Equation \eqref{auxiliary2} implies
\begin{equation}
\label{eq:318}
	v_\e(x,\omega) \in \alpha(u_\e(x,\omega),T_{x/\e}\omega), \quad \textrm{for a.e. }(x,\omega)\in D\times \Omega.
\end{equation}
By Birkhoff's Theorem (and \eqref{birk:weak}, in particular), for a.e. $\o \in \O$ we have 
\begin{equation}
\label{weak}
	u_\e(\cdot,\o) \rightharpoonup \mathbb E(u)\quad \textrm{weakly in }L^p(D;\R^n), \qquad v_\e(\cdot,\o) \rightharpoonup \mathbb E(v) \quad \textrm{weakly in }L^q(D;\R^n).
\end{equation}
Since $\alpha$ is monotone, by \eqref{e:inclusion} and \eqref{eq:318}, for a.e. $\omega \in \O$
\begin{equation}
\label{ineq}
	\int_D (E_\o^\e(x) - v_\e(x,\o))\cdot (J_\o^\e(x) - u_\e(x,\o))\phi(x)\,dx \ge 0,
\end{equation}
for any $\phi\in C_c^\infty(D)$ with $\phi\ge 0$. Since $u_\e$ is solenoidal and $v_\e$ is potential, by \eqref{e:bound}
\begin{align*}
	\{\Curl(E_\o^\e - v_\e(\cdot,\o))\}_\e\text{ is compact in }W^{-1,q}(D),\\
	  \{\Div(J_\o^\e - u_\e(\cdot,\o))\}_\e\text{ is compact in }W^{-1,p}(D).
\end{align*}
By \eqref{e:convergence}, \eqref{weak}, and Lemma \ref{lemma:divcurl}, we can thus pass to the limit as $\e\to 0$ in \eqref{ineq}: 
\[
	\int_D (E_\o^0(x) -\mathbb E( v))\cdot (J_\o^0(x) -\mathbb E( u))\phi(x)\,dx \ge 0, \quad \textrm{for a.e. }\omega \in \O.
\]
Since the last inequality holds for all nonnegative $\phi\in C_c^\infty(D)$, it holds also pointwise, for almost every $x \in D$:
\[
	(E_\o^0(x) -\mathbb E( v))\cdot (J_\o^0(x) -\mathbb E( u)) \ge 0, \quad \textrm{for a.e. }\omega \in \O.
\]
To conclude, since $\mathbb{E}(u)=\xi$, $\mathbb{E}(v)=\eta$ are arbitrary vectors in $\mathcal G_{\alpha_0}$, the maximal monotonicity of $\alpha_0$ implies that
\[
	E_\o^0(x) \in \alpha_0(J_\o^0(x))
\]
for a.e. $x\in D$, $\o\in \O$. 
\end{proof}
\begin{oss}
In this section's results, the function spaces $L^p_\sol(\O)$ and $L^q_\pot(\O)$ can be generalized to a couple of nonempty, closed and convex sets 
\[
	\mathcal U\subset L^p(\O;\R^n), \quad \mathcal V\subset L^q(\O;\R^n)
\]
such that
\[
	\mathbb E(u \cdot v) = \mathbb E(u)\cdot \mathbb E(v), \quad  \forall (u,v)\in \mathcal U \times \mathcal V.
\]
Furthermore, proposition \ref{prop-f0} and Lemma \ref{lemma:problema} remain valid if the previous inequality is weakened to
\[
	\mathbb E(u \cdot v) \geq \mathbb E(u)\cdot \mathbb E(v), \quad  \forall (u,v)\in \mathcal U \times \mathcal V.
\]
\end{oss}

\section{Applications}

\subsection{The Ohm-Hall model}
\label{sec:ohm}

In this paragraph we address the homogenization problem for the Ohm-Hall model for an electric conductor. For further information about the Ohm-Hall effect we refer the reader to \cite[pp. 11-15]{AshcroftMermin}, \cite[Section 22]{LandauLifshitz}.
We consider a non homogeneous electric conductor, that occupies a bounded Lipschitz domain $D \subset \R^3$ and is subjected to a magnetic field. We assume that the electric field $E$ and the current density $J$ fulfill the constitutive law
\begin{equation}
\label{eq:ohmhall}
	E(x,t) = \alpha(J(x,t),x) +h(x)J(x,t) \times B(x,t) + E_a(x),
\end{equation}
where $\alpha(\cdot,x):\R^3 \mapsto \R^3$ is a (single-valued) maximal monotone mapping for a.e.\,$x\in D$, $B$ is the magnetic induction field, $h$ is the (material dependent) Hall coefficient, and $E_a$ is an applied electromotive force. We couple \eqref{eq:ohmhall} with the Faraday law and with the stationary law of charge-conservation:
\begin{align*}
	\Curl E &= - \frac{\partial B}{\partial t},\\
	\Div J &=0.
\end{align*}	
Following \cite{VisintinAsAn2013}, we assume that $h,B,E_a$ are given, we deal with the stationary system, thus dropping the time variable, and we define the maximal monotone operator $\beta(\cdot,x):\R^3 \mapsto \R^3$ and the vector field $g:D \to \R^3$
\[
	\beta(J,x):=\alpha(J,x) +h(x)J \times B(x) + E_a(x),\qquad g(x):=- \frac{\partial B}{\partial t}(x).
\]
A single-valued parameter-dependent operator $\beta$ is \textit{strictly monotone uniformly in} $x$, if there exists $\theta>0$ such that for a.e. $x\in D$
\begin{equation}
\label{eq:smon}
	(\beta(v_1,x) - \beta(v_2,x))\cdot (v_1-v_2)   \geq \theta{\| v_1-v_2\|}^2\quad \forall\,v_1,v_2\in \R^3.
\end{equation}

The following existence and uniqueness result is a classical consequence of the maximal monotonicity of $\alpha$ (see, e.g., \cite{Showalter97, VisintinAsAn2013}).

\begin{theorem}
\label{thV}
Let $D \subset \R^3$ be a bounded Lipschitz domain. Let $\{\beta(\cdot,x)\}_{x\in D}$ be a family of single-valued maximal monotone operators on $\R^3$. Assume moreover that there exist constants $a,c>0$ and $b\geq 0$ such that for a.e.  $x\in D,\,\forall v \in \R^3$
\begin{align}
	|\beta(x,v)| &\leq c(1+|v|), \label{eq:bounded} \\
	\beta(x,v)\cdot v &\geq a|v|^2 -b. \label{eq:coercivity}
\end{align}
Let $g \in L^2(D;\R^3)$ be given, such that $\nabla \cdot g=0$, distributionally. Then, there exists $E,J\in L^2(D;\R^3)$ such that
\begin{equation}
\label{eq:estimates}
	{\|E\|}_{L^2} +{\|J\|}_{L^2}\leq C\left(1+{\|g\|}_{L^2}\right)
\end{equation}
and, denoting by $\nu$ the outward unit normal to $\partial D$,
\begin{align}
	 E(x) &= \beta(J(x),x)			& &\text{in }D, \label{P:incl}\\
	 \nabla \times E(x) &= g(x) 	& &\text{in }D, \label{P:ele}\\
	 \nabla \cdot J(x)&=0 			& &\text{in }D, \label{P:magn}\\
	 E(x) \times \nu(x) &=0				& &\text{on }\partial D. \label{P:bound}
\end{align}
Moreover, if $\beta$ is strictly monotone uniformly in $x\in D$, then the field $J$ is uniquely determined, while if $\beta^{-1}$ is strictly monotone uniformly in $x\in D$, then the field $E$ is uniquely determined.
\end{theorem}

\begin{remark} Conditions \eqref{P:ele}--\eqref{P:magn} have to be intended in the weak sense -- see below -- while \eqref{P:bound} holds in $H^{-1/2}(\partial D;\R^3)$. Note also that the divergence of $g$ vanishes naturally according to \eqref{P:ele}. 
\end{remark}

Let $(\O,\mathcal{A},\mu)$ be a probability space endowed with a 3-dimensional ergodic dynamical system $T_x \colon \O \to \O$, with $x\in \R^3$. Let $\{\alpha(\cdot,\o)\}_{\o\in \O}$ be a family of maximal monotone operators on $\R^3$, and let 
\begin{equation}
\label{hyp:data}
	h \in L^\infty(\O), \quad B \in L^\infty(\O;\R^3), \quad E_a\in L^2(\O;\R^3).
\end{equation}
For any $J \in \R^3$ and for any $(x,\o)\in D\times\O$ let 
\begin{equation}
\label{hyp:beta}
	\beta(J,\o):=\alpha(J,\o)+h(\o)J \times B(\o)+E_a(\o).
\end{equation}

In order to apply the scale integration procedure, we assume that 
\begin{equation}
\label{hyp:fsc}
	\text{the representative function $f$ of $\beta$ is coercive, in the sense of \eqref{coercivity},}
\end{equation}
moreover, to ensure uniqueness of a solution $(E,J)$, we assume that
\begin{equation}
\label{hyp:smon}
	\beta\text{ and }\beta^{-1} \text{ are strictly monotone, uniformly with respect to }x\in D.
\end{equation}
As in the previous section $\beta_0$ stands for the maximal monotone operator represented by $f_0$ given by \eqref{f0}. For any $\e>0$ define 
\[
	\beta_\e(\cdot,x,\o):=\beta(\cdot,T_{x/\e}\o).
\]
Then $\{\beta_\e(\cdot,x,\o)\}_{(x,\o)\in D\times \O}$ is a family of maximal monotone operators on $\R^3$. Let $g_\e \in L^2(D\times \O;\R^3)$ with $g_\e \rightharpoonup g$ in $L^2(D;\R^3)$ for some $g \in L^2(D;\R^3)$, for a.e. $\o\in\O$; assume that 
\begin{equation}
\label{hyp:divge}
	\nabla \cdot g_\e=0, \quad \text{in }\mathscr D'(D), \text{ for a.e. }\o\in\O.
\end{equation} 

We are ready to state and prove the homogenization result for the Ohm-Hall model.

\begin{theorem}\label{ohm-hall}
Assume that \eqref{hyp:data}--\eqref{hyp:divge} are fulfilled. Then

1. For $\mu$-a.e. $\o\in\O$, for any $\e>0$ there exists $(E_\o^\e,J_\o^\e)\in L^2(D;\R^3) \times L^2(D;\R^3)$ such that 
\begin{align}
	& E_\o^\e(x) = \beta_\e(J_\o^\e(x),x,\o)			& &\text{in }D, \label{P:incl-eps}\\
	& \nabla \times E_\o^\e(x) = g_\e(x,\o) 	& &\text{in }D, \label{P:ele-eps}\\
	& \nabla \cdot J_\o^\e(x)=0 			& &\text{in }D, \label{P:magn-eps}\\
		& E_\o^\e(x) \times \nu(x) =0				& &\text{on }\partial D. \label{P:bound-eps}
\end{align}

2. There exists $(E,J)\in L^2(D;\R^3) \times L^2(D;\R^3)$ such that, up to a subsequence, 
\begin{equation}
\label{eq:conv}
	E_\o^\e \rightharpoonup E\quad \text{and}\quad J_\o^\e \rightharpoonup J 
\end{equation}
as $\e\to0$, weakly in $L^2(D;\R^3)$. 

3. The limit couple $(E,J)$ is a weak solution of
\begin{align}
	& E(x) = \beta_0(J(x))			& &\text{in }D, \label{P:incl-hom}\\
	& \nabla \times E(x) = g(x) 	& &\text{in }D, \label{P:ele-hom}\\
	& \nabla \cdot J(x)=0 			& &\text{in }D, \label{P:magn-hom}\\
	& E(x) \times \nu(x) =0				& &\text{on }\partial D. \label{P:bound-hom}
\end{align}
\end{theorem} 

\begin{proof}
1. Assumption \eqref{hyp:smon} implies that $\beta$ is single valued and that almost every realization $(x,v)\mapsto \beta(v,T_x\o)$ satisfies the boundedness and coercivity assumptions \eqref{eq:bounded} and \eqref{eq:coercivity}. Therefore, by Theorem \ref{thV} for almost any $\o\in\O$ and for any $\e>0$ problem \eqref{P:incl-eps}--\eqref{P:bound-eps} has a unique solution.

2. Let $\o\in \O$ be fixed. By \eqref{eq:estimates} the families $\{E_\o^\e\}_\e$ and $\{J_\o^\e\}_\e$ are weakly relatively compact in $L^2(D;\R^3)$, therefore, there exist a subsequence $\e_n \to 0$ and a couple $(E_\o,J_\o)\in L^2(D;\R^3) \times L^2(D;\R^3)$ satisfying \eqref{eq:conv}. A priori, $(E_\o,J_\o)$ depends on $\o\in \O$.

3.  The weak formulation of \eqref{P:ele-eps}--\eqref{P:bound-eps} is: 
\begin{equation}
\label{eq:weak}
	\int_D E_\o^\e \cdot (\nabla \times \phi) + J_\o^\e \cdot \nabla \psi\,dx = \int_D g_\e \cdot \phi\, dx,
\end{equation}
for all $\phi\in \{H^1(D;\R^3): \phi \times \nu=0 \text{ on }\partial D\}$, for all $\psi\in H^1_0(D;\R)$. Passing to the limit in \eqref{eq:weak}, one gets
\[
	\int_D E_\o \cdot (\nabla \times \phi) + J_\o \cdot \nabla \psi\,dx = \int_D g \cdot \phi\, dx,
\]	
which is exactly the weak formulation of \eqref{P:ele-hom}--\eqref{P:bound-hom}. Equations \eqref{P:ele-eps} and \eqref{P:magn-eps} imply that $\{E_\o^\e\}_\e$ and $\{J_\o^\e\}_\e$ satisfy also the div-curl compactness condition \eqref{e:bound}. Therefore, we can apply the abstract stochastic homogenization Theorem \ref{newmain}, which yields 
\[
	E_\o(x) = \beta_0(J_\o(x)).
\]
We have thus proved that $(E_\o,J_\o)$ is a weak solution of \eqref{P:incl-hom}--\eqref{P:bound-hom}. In order to conclude we have to eliminate the dependence on $\o\in \O$.

4. By Lemma \ref{lemma:strictmon} and assumption \eqref{hyp:smon}, $\beta_0$ and $\beta_0^{-1}$ are strictly monotone, and therefore \eqref{P:incl-hom}--\eqref{P:bound-hom} admits a unique solution. Thus, $(E,J):=(E_\o,J_\o)$ is independent of $\o \in \O$. 
\end{proof}

\subsection{Nonlinear elasticity}
\label{ssec:nle}

Another straightforward application of the homogenization theorem \ref{newmain} is given in the framework of deformations in continuum mechanics (see, e.g., \cite[Chapter 3]{Ciarlet}). Elastic materials are usually described through the deformation vector $u:D\times(0,T) \to \R^3$ and the stress tensor $\sigma:D\times(0,T) \to \R^{3 \times 3}_s$. Here $D\subset \R^3$ is the spatial domain and $\R^{3 \times 3}_s$ the space of symmetric 3x3 matrices. We assume  the following constitutive relation relating stress and deformation:
\begin{equation}
\label{eq:nlelastic}
	\sigma(x,t) = \beta(\nabla u(x,t),x),
\end{equation}
where $\beta(\cdot,x):\R^{3\times 3} \mapsto \R^{3\times 3}$ is a (single-valued) maximal monotone mapping for a.e.\,$x\in D$. We couple \eqref{eq:nlelastic} with the conservation of linear momentum:
\begin{equation*}
	\rho\partial_t^2 u - \nabla \cdot \sigma = f,
\end{equation*}	
where $\rho$ is the density and $f$ represents the external forces. For sake of simplicity, we choose to deal with the stationary system only and we set $\rho\partial_t^2 u=0$.

 The following existence and uniqueness result is a classical consequence of the maximal monotonicity of $\beta$ (see, e.g., \cite{Evans02,Showalter97}).

\begin{theorem}
\label{thV2}
Let $D \subset \R^3$ be a bounded Lipschitz domain. Let $\{\beta(\cdot,x)\}_{x\in D}$ be a family of single-valued maximal monotone operators on $\R^{3\times 3}$ that satisfy \eqref{eq:bounded} and \eqref{eq:coercivity}. Let $f \in L^2(D;\R^3)$ be given. Then, there exists $\sigma\in L^2(D;\R^{3\times 3})$ and $u\in H^1_0(D;\R^3)$ such that
\begin{equation}
\label{Q:estimates}
	{\|u\|}_{H^1} +{\|\sigma\|}_{L^2}\leq C\left(1+{\|f\|}_{L^2}\right)
\end{equation}
and, denoting by $\nu$ the outward unit normal to $\partial D$,
\begin{align}
	 \sigma(x) &= \beta(\nabla u(x),x)			& &\text{in }D, \label{Q:incl}\\
	 -\nabla \cdot \sigma(x) &= f(x) 				& &\text{in }D, \label{Q:ele}\\
	 u(x) &=0				& &\text{on }\partial D. \label{Q:bound}
\end{align}
Moreover, if $\beta$ is strictly monotone uniformly in $x\in D$, then  $u$ is uniquely determined, while if $\beta^{-1}$ is strictly monotone uniformly in $x\in D$, then $\sigma$ is uniquely determined.
\end{theorem}
As above, we consider a family of maximal monotone operators $\{\beta(\cdot,\o)\}_{\o\in \O}$ on $\R^{3\times 3}$,  $\beta_0$ stands for the maximal monotone operator represented by $f_0$, and for any $\e>0$  
\[
	\beta_\e(\cdot,x,\o):=\beta(\cdot,T_{x/\e}\o)
\]
defines a family of maximal monotone operators on $\R^{3\times 3}$. Let $f_\e \in L^2(D\times \O;\R^3)$ with $f_\e \rightharpoonup f$ in $L^2(D;\R^3)$ for some $f \in L^2(D;\R^3)$, for a.e. $\o\in\O$. The correspondent homogenization theorem is the following.

\begin{theorem}
\label{mech}
Assume that \eqref{hyp:fsc} and \eqref{hyp:smon} are fulfilled. Then

1. For $\mu$-a.e. $\o\in\O$, for any $\e>0$ there exist $(u_\o^\e,\sigma_\o^\e)\in H^1_0(D;\R^3) \times L^2(D;\R^3)$ such that 
\begin{align}
	& \sigma_\o^\e(x) = \beta_\e(\nabla u_\o^\e(x),x,\o)			& &\text{in }D, \label{Q:incl-eps}\\
	& -\nabla \cdot \sigma_\o^\e(x) = f_\e(x,\o) 	& &\text{in }D, \label{Q:ele-eps}\\
	& u_\o^\e(x)  =0				& &\text{on }\partial D. \label{Q:bound-eps}
\end{align}

2. There exist $(u,\sigma)\in H^1_0(D;\R^3) \times L^2(D;\R^3)$ such that, up to a subsequence, 
\begin{equation}
\label{Q:conv}
	u_\o^\e \rightharpoonup u\quad \text{and}\quad \sigma_\o^\e \rightharpoonup \sigma 
\end{equation}
as $\e\to0$, weakly in $H^1(D;\R^3)$ and $L^2(D;\R^3)$, respectively. 

3. The limit couple $(u,\sigma)$ is a weak solution of
\begin{align}
	& \sigma(x) = \beta_0(\nabla u(x))			& &\text{in }D, \label{Q:incl-hom}\\
	& -\nabla \cdot \sigma(x) = f(x) 	& &\text{in }D, \label{Q:ele-hom}\\
	& u(x) =0				& &\text{on }\partial D. \label{Q:bound-hom}
\end{align}
\end{theorem} 
\begin{proof}
Steps 1. and 2. follow exactly as in the proof of Theorem \ref{ohm-hall}.  

3.  The weak formulation of \eqref{Q:ele-eps}--\eqref{Q:bound-eps} is the following: 
\begin{equation}
\label{Q:weak}
	\int_D \sigma_\o^\e \cdot \nabla \phi\,dx = \int_D f_\e  \phi\, dx,
\end{equation}
for all $\phi\in H^1_0(D)$. Passing to the limit as $\e \to 0$, one gets
\[
	\int_D \sigma_\o \cdot \nabla \phi\,dx = \int_D f  \phi\, dx,
\]	
which is exactly the weak formulation of \eqref{Q:ele-hom}--\eqref{Q:bound-hom}. Equation \eqref{Q:ele-eps} and estimate \eqref{Q:estimates} imply that $\{\sigma_\o^\e\}_\e$ and $\{\nabla u_\o^\e\}_\e$ satisfy also the div-curl compactness condition \eqref{e:bound}
\[	
	\{ \Div \sigma_\o^\e\}_{\e \geq 0} \text{ is compact in }W^{-1,2}(D;\R^3),\quad 	
		\{ \Curl \nabla u_\o^\e\}_{\e \geq 0} \text{ is compact in }W^{-1,2}(D;\R^{3 \times 3}).
\]
Therefore, we can apply the abstract stochastic homogenization Theorem \ref{newmain}, (with $\sigma$ in place of $J$ and $\nabla u$ in place of $E$), which yields 
\[
	\sigma_\o(x) = \beta_0(\nabla u_\o(x)).
\]
Finally, the strict monotonicity of the limit operators $\beta_0$ and $\beta_0^{-1}$ yields uniqueness and therefore independence of $\o$ for the solution $(u,\sigma)$.
\end{proof}

\bibliographystyle{abbrv}

\end{document}